\documentclass{article}

\usepackage{arxiv}

\usepackage[utf8]{inputenc} % allow utf-8 input
\usepackage[T1]{fontenc}    % use 8-bit T1 fonts
\usepackage{hyperref}       % hyperlinks
\usepackage{url}            % simple URL typesetting
\usepackage{booktabs}       % professional-quality tables
\usepackage{amsfonts}       % blackboard math symbols
\usepackage{nicefrac}       % compact symbols for 1/2, etc.
\usepackage{microtype}      % microtypography
\usepackage{lipsum}
\usepackage{fancyhdr}       % header
\usepackage{graphicx}       % graphics
\graphicspath{{media/}}     % organize your images and other figures under media/ folder
\usepackage{amsmath}
\usepackage{amssymb}
\usepackage{amsthm}
\usepackage{stackengine} 
\stackMath

\newtheorem{theorem}{Theorem}

\newtheorem*{theorem*}{Theorem}

\title{The secret life of matrix factorizations: how matrix decompositions reveal and keep secrets of linear equations and what we can do about it
}

\author{
  Michał P. Karpowicz\\
  NASK National Research Institute\\
  Warsaw, Poland\\
  \texttt{michal.karpowicz@nask.pl}\\
}

\renewcommand{\today}{}

\begin{document}
\maketitle

\begin{abstract}

This paper explores the relationship between matrix factorizations and linear matrix equations. It shows that every matrix factorization defines two hidden projectors, one for the column space and one for the row space of a matrix, and how to calculate them. The projectors can be applied to solve linear matrix equations, generate low-rank approximations, or design randomized matrix algorithms. But also, as demonstrated, they can be applied in cryptography to encrypt and decrypt messages. The paper discusses some of the security implications of this application and leaves some questions open for further investigation. The basic concepts are illustrated with source code listings. Finally, this work shares some personal reflections on the meaning and importance of understanding in the time of the artificial intelligence revolution.

\end{abstract}

% keywords can be removed
% \keywords{First keyword \and Second keyword \and More}
% \MSC: 15A10, 15A23, 15A24, 94A60, 68P25

\section{Introduction}

Matrix factorization, the idea that we can decompose a matrix into a product of other matrices, plays a leading role in mathematics, science, and engineering. I think there are at least two good reasons for that. First, matrix factorization explains how numbers stored in columns and rows of a matrix relate to each other. It allows us to count and identify columns and rows that are linearly independent and show how to reconstruct the remaining dependent ones. It can even tell us which columns and rows are more important than others. It helps us understand the time evolution of dynamical systems, conservation laws in physics, internal structures of networks, components of signals, and the list goes on and on. 

Second, matrix factorization algorithms, or matrix factorizations for short, help us learn from data, and learning usually means solving linear equations. Even when dealing with nonlinear problems, we numerically solve linear equations in the background, which involves what we can learn about the problem from matrix factorizations. There exists an intimate relationship between matrix factorizations and linear matrix equations.

What is the nature of the relationship that makes all those tricks possible? As we will see, every matrix factorization defines two hidden projectors: one for the column space and one for the row space of a matrix. The first projects the original columns onto the column space, and the second projects the original rows onto the row space of the original matrix. That is why the product of factorization factors reconstructs the matrix and enables low-rank approximations when properly designed.

In this paper, I explore the origins and structure of these hidden projectors. To explain the key ideas better, I start with elementary concepts, illustrate them with simple numerical examples, and then move to more general ones. Also, from many potential applications, including the design of new factorizations, low-rank approximations, and randomized optimization algorithms, I present one in cryptography that I find intriguing, important, and inspiring. The projectors hidden in matrix factorizations can be used to hide secrets. We will design a cryptographic system encrypting and decrypting messages and address the challenging security questions while keeping others open for further investigation.

\paragraph{Notation.} Capital bold letters denote matrices, $\mathbf{W}^*$ denotes the conjugate transpose of $\mathbf{W}$, and $\mathbf{W}^+$ denotes the Moore-Penrose pseudoinverse of $\mathbf{W}$. We consider the case of $\mathbf{W}\in\mathbb{K}^{m\times n}$, where $\mathbb{K}$ is either $\mathbb{R}$ or $\mathbb{C}$. The lowercase bold letter $\mathbf{w}$ denotes a column vector, whereas $\mathbf{w}^*$ denotes a row vector. The $k$ by $k$ identity matrix is denoted as $\mathbf{I}_k$. Given ordered subindex sets $I$ and $J$, $\mathbf{A}(I,J)$ denotes the submatrix of $\mathbf{A}$ containing rows and columns of $\mathbf{A}$ indexed by $I$ and $J$, with $\mathbf{A}(:,J)$ extracting the columns of $\mathbf{A}$ indexed by $J$. With $k$ being a positive integer, $1{:}k$ denotes the ordered set $(1,\dots,k)$. 

\paragraph{Related work.} The results presented in this paper explain the results introduced in \cite{KarpowiczMFT} and extend them by introducing generalizations and showing practical applications.

Great and comprehensive references on the matrix factorization theory include Golub and Van Loan \cite{golub2013matrix}, Stewart \cite{stewart1998matrix}, Demmel \cite{demmel1997applied}, Trefethen \cite{trefethen1997numerical} and Strang \cite{strangIntroLA}. For more details on recent results introducing concepts of randomized linear algebra, see Nakatsukasa et al. \cite{nakatsukasa2020fast,nakatsukasa2021accurate}, Hamm et al. \cite{HAMM2023236,hamm2020perspectives,cai2021robust}, Martinsson and Tropp \cite{martinsson2019randomized,martinsson2020randomized} or Kannan and Vempala \cite{kannan2017randomized}.

Schneier and co-authors present a wonderful introduction to applied cryptography in  \cite{schneier2007applied,ferguson2011cryptography}. Menezes, Van Oorschot, and Vanstone give a comprehensive overview of mathematics involved in \cite{menezes2018handbook}, whereas Singh \cite{singh2000code} and Holden \cite{holden2017secrets} present an excellent historical perspective.

\section{Main result}

The following is the main result of this paper. 
It shows how to reconstruct a rank-k matrix from a product of matrices of possibly higher rank that are possibly rank-deficient. The representation depends on parameters that can be selected on-demand subject to rank-preserving conditions.

\begin{theorem*}[Theorem \ref{thm:MFT-LME} of Section~\ref{sec:Matrix decompositions}]
Consider rank-$k$ matrices $\mathbf{A}\in\mathbb{K}^{m\times n}$, $\mathbf{F}\in\mathbb{K}^{m\times r}$ and $\mathbf{H}\in\mathbb{K}^{m\times q}$, where $k\le\min\{m,n\}$, $k\le r$ and $k\le q$, such that $\mathcal{C}(\mathbf{A}) \subset \mathcal{C}(\mathbf{F})$ and $\mathcal{C}(\mathbf{A}^*) \subset \mathcal{C}(\mathbf{H})$. Also, suppose that $\mathbf{B}\in\mathbb{K}^{m\times r}$ and $\mathbf{D}\in\mathbb{K}^{n\times q}$ meet the following rank-preserving condition:
\begin{align}
\mathrm{rank}(\mathbf{B}^*\mathbf{F}) = \mathrm{rank}(\mathbf{H}^*\mathbf{D}) = k.
\end{align}

If 
\begin{align}
\mathbf{Y}^* = (\mathbf{B}^*\mathbf{F})^+\mathbf{B}^*
\quad\text{and}\quad
\mathbf{X} = \mathbf{D}(\mathbf{H}^*\mathbf{D})^+,
\end{align}
then $\mathbf{A}$ can be factorized in the form
\begin{align}
\mathbf{A} = \mathbf{F}\mathbf{G}\mathbf{H}^*,
\end{align}
where
\begin{align}
\mathbf{G} = 
\mathbf{Y}^*
\mathbf{A}
\mathbf{X}
+ 
\mathbf{W} - \mathbf{Y}^*\mathbf{F}\mathbf{W}\mathbf{H}^*\mathbf{X}
\end{align}
and $\mathbf{W}$ is arbitrary.
\end{theorem*}

In the following sections, I explain each statement of this result step by step, starting from elementary concepts and moving to more general ones. First, in Section~\ref{sec:One form to rule them all}, I introduce a reduced form of matrix factorization, a model to investigate all forms of matrix factorizations. Next, I study its internal structure in Sections~\ref{sec:The projection equation} and \ref{sec:Hidden projectors}. That study leads to the concept of projectors hidden within the structure of the reduced form. The projectors depend on two matrices, $\mathbf{Y}^*$ and $\mathbf{X}$, which can take the following general and convenient form:
\begin{align}
\mathbf{Y}^* = (\mathbf{B}^*\mathbf{F})^+\mathbf{B}^*
\quad\text{and}\quad
\mathbf{X} = \mathbf{D}(\mathbf{H}^*\mathbf{D})^+.
\end{align}
Finally, in Sections~\ref{sec:Matrix decompositions} I show how to design and use hidden projectors to reconstruct an arbitrary matrix. The design of encryption and decryption functions of a cryptographic algorithm is presented in Section~\ref{sec:Hiding secrets with hidden projectors}.

\section{One form to rule them all}
\label{sec:One form to rule them all}

We ask general questions about matrix factorizations, so we need a model to investigate all forms of matrix factorizations at once. Let us call it a general reduced form of matrix factorizations. It represents an $m$ by $n$ and rank-$k$ matrix $\mathbf{A}$ as a~product of three matrices:
\begin{align}\label{eq:similarity}
\setstackgap{L}{15pt}\def\stacktype{L}\def\sz{\scriptstyle}
\begin{aligned}
\stackunder{\mathbf{A}}{\sz (m\times n)} &= 
\stackunder{\mathbf{F}}{\sz (m\times k)} &&
\stackunder{\mathbf{G}}{\sz (k\times k)} &&
\stackunder{\mathbf{H}^*.}{\sz (k\times n)}
\end{aligned}
\end{align}
It is instructive to think of some examples, so let us first take a look at few famous ones.

\paragraph{Similarity transformations.} Two square matrices $\mathbf{A}$ and $\mathbf{B}$ are similar, if 
\begin{align}\label{eq:genfactorization}
\setstackgap{L}{15pt}\def\stacktype{L}\def\sz{\scriptstyle}
\begin{aligned}
\stackunder{\mathbf{A}}{\sz (k\times k)} &= 
\stackunder{\mathbf{M}}{\sz (k\times k)} &&
\stackunder{\mathbf{B}}{\sz (k\times k)} &&
\stackunder{\mathbf{M}^{-1}.}{\sz (k\times k)}
\end{aligned}
\end{align}
The equation holds, if we can find a~nonsingular matrix $\mathbf{M}$ to define the transformation above. For example, 
\begin{align}
\begin{aligned}
\mathbf{A} = 
\begin{bmatrix}
1 & 2\\ 3 & 4
\end{bmatrix}=
\begin{bmatrix}
0 & 1\\ 1 & 1
\end{bmatrix}
\begin{bmatrix}
2 & 4\\ 2 & 3
\end{bmatrix}
\begin{bmatrix}
-1 & 1\\ \phantom{-}1 & 0
\end{bmatrix}=
\mathbf{M}\mathbf{B}\mathbf{M}^{-1}.
\end{aligned}
\end{align}
Prominent similarity transformations include eigendecomposition, Schur decomposition and Jordan canonical form. Interpreted as a matrix factorizations, they match the reduced form with $m = n = k$, $\mathbf{F}=\mathbf{M}$, $\mathbf{G}=\mathbf{B}$, and $\mathbf{H}^*=\mathbf{M^{-1}}$. 

\paragraph{Reduced SVD.} The singular value decomposition (SVD) shows that every matrix $\mathbf{A}$ is isometric to the singular value matrix 
\begin{align}
\setstackgap{L}{15pt}\def\stacktype{L}\def\sz{\scriptstyle}
\begin{aligned}
\stackunder{\mathbf{S}}{\sz (m\times n)} = 
\stackunder{\mathbf{U}^*}{\sz (m\times m)} &&
\stackunder{\mathbf{A}}{\sz (m\times n)} &&
\stackunder{\mathbf{V}.}{\sz (n\times n)}
\end{aligned}
\end{align}
Matrices $\mathbf{U}$ and $\mathbf{V}$ are complex unitary, and $\mathbf{S}$ is a diagonal matrix with (real non-negative) singular values on the diagonal. The reduced version of the SVD factors a~rank-$k$ matrix, $k < \min\{m,n\}$, into the product:
\begin{align}
\setstackgap{L}{15pt}\def\stacktype{L}\def\sz{\scriptstyle}
\begin{aligned}
\stackunder{\mathbf{A}}{\sz (m\times n)} &=& 
\stackunder{\mathbf{U}(:,1{:}k)}{\sz (m\times k)} &&
\stackunder{\mathbf{S}(1{:}k,1{:}k)}{\sz (k\times k)} &&
\stackunder{\mathbf{V}(:,1{:}k)^*}{\sz (k\times n)}
= \mathbf{U}_k\mathbf{S}_k\mathbf{V}_k^*.
\end{aligned}
\end{align}

It is easy to see how it matches the general reduced form. 

The following example shows how a rectangular matrix $\mathbf{A}$ can be diagonalized with two unitary matrices:
\begin{align}
\begin{aligned}
\mathbf{A} = 
\begin{bmatrix}
\phantom{-}1 & \phantom{-}1 \\ 
\phantom{-}1 & -1 \\ 
\phantom{-}1 & \phantom{-}1
\end{bmatrix}=
\begin{bmatrix}
-\sqrt{2}/2 & \phantom{-}0\\ \phantom{-}0 & -1 \\ -\sqrt{2}/2 & \phantom{-}0
\end{bmatrix}
\begin{bmatrix}
2 & 0\\ 0 & \sqrt{2}
\end{bmatrix}
\begin{bmatrix}
-\sqrt{2}/2 & -\sqrt{2}/2\\ -\sqrt{2}/2 & \sqrt{2}/2
\end{bmatrix}=
\mathbf{U}_k\mathbf{S}_k\mathbf{V}_k^*
\end{aligned}
\end{align}

\paragraph{Column pivoted QR.} The column pivoted QR (CPQR) decomposition factors any rectangular matrix into the product of an orthogonal $m$ by $m$ matrix $\mathbf{Q}$, an upper triangular $m$ by $n$ matrix $\mathbf{R}$ and a rectangular $n$ by $n$ permutation matrix $\mathbf{\Pi}$. The thin CPQR decomposition of a~rank-$k$ matrix takes $k$ columns of $\mathbf{Q}$ and non-zero rows of $\mathbf{R}$, so that
\begin{align}
\setstackgap{L}{15pt}\def\stacktype{L}\def\sz{\scriptstyle}
\begin{aligned}
\stackunder{\mathbf{A}}{\sz (m\times n)} &=& 
\stackunder{\mathbf{Q}(:,1{:}k)}{\sz (m\times k)} &&
\stackunder{\mathbf{R}(1{:}k,:)}{\sz (k\times n)} &&
\stackunder{\mathbf{\Pi}^*.}{\sz (n\times n)}
\end{aligned}
\end{align}
We get the general reduced form with
\begin{align}
\begin{aligned}
\mathbf{F} = \mathbf{Q}(:,1{:}k)
\quad\text{and}\quad
\mathbf{G} = \mathbf{I}_k,
\quad\text{and}\quad
\mathbf{H}^* = \mathbf{R}(1{:}k,:)\mathbf{\Pi}^*,
\end{aligned}    
\end{align}
which is also illustrated below:
\begin{align}
\begin{aligned}
\mathbf{A} &= 
\begin{bmatrix}
\phantom{-}1 & \phantom{-}1 \\ 
\phantom{-}0 & -1 \\ 
\phantom{-}0 & \phantom{-}0
\end{bmatrix} =
\begin{bmatrix}
-\sqrt{2}/2 & \phantom{-}\sqrt{2}/2\\ \phantom{-}\sqrt{2}/2 & \phantom{-}\sqrt{2}/2 \\ 0 & 0
\end{bmatrix}
\begin{bmatrix}
1 & 0\\ 0 & 1
\end{bmatrix}
\begin{bmatrix}
-\sqrt{2} & -\sqrt{2}/2\\ 0 & \phantom{-}\sqrt{2}/2
\end{bmatrix}
\begin{bmatrix}
0 & 1\\ 1 & 0
\end{bmatrix} 
\\
&=
\mathbf{Q}(:,1{:}k)
\mathbf{I}_k
\left(
\mathbf{R}(1{:}k,:)\mathbf{\Pi}^*
\right).
\end{aligned}
\end{align}

\paragraph{LU decomposition.} The idea behind this famous decomposition is to transform a square matrix $\mathbf{A}$ into an upper-triangular matrix $\mathbf{U}$ by subtracting multiples of each row from subsequent rows. The operation eliminates entries below the diagonal and is equivalent to multiplying $\mathbf{A}$ by a~lower-triangular matrix $\mathbf{L}^{-1}$. However, to control the condition numbers of $\mathbf{L}$ and $\mathbf{U}$ and avoid zeros along the new diagonal, it is usually necessary to change the row ordering of $\mathbf{A}$. We can encode the row exchanges into a~permutation matrix $\mathbf{\Pi}$. These operations yield
\begin{align}
\begin{aligned}
\mathbf{A} = \mathbf{\Pi}^*\mathbf{L}\mathbf{U}.
\end{aligned}
\end{align}
We get the general reduced form for decomposition by setting $\mathbf{F}=\mathbf{\Pi}^*$, $\mathbf{G}=\mathbf{L}$, and $\mathbf{H}^*=\mathbf{U}$. The following example illustrates the idea: 
\begin{align}
\begin{aligned}
\mathbf{A} = 
\begin{bmatrix}
0 & 1 & 1\\ 1 & 2 & 1 \\2 & 7 & 9
\end{bmatrix}=
\begin{bmatrix}
0 & 0 & 1\\ 0 & 1 & 0\\ 1 & 0 & 0
\end{bmatrix}
\begin{bmatrix}
1 & 0 & 0\\ 1/2 & 1 & 0 \\ 0 & -2/3 & 1
\end{bmatrix} 
\begin{bmatrix}
2 & 7 & 9\\ 0 & -3/2 & -7/2 \\ 0 & 0 & -4/3
\end{bmatrix} =
\mathbf{\Pi}^*\mathbf{L}\mathbf{U}.
\end{aligned}
\end{align}

\paragraph{CUR decomposition.}
The CUR decomposition provides an approximation:
\begin{align}\label{eq:CUR}
\mathbf{A} \approx 
\mathbf{C}\mathbf{U}\mathbf{R},
% =\mathbf{C}(\mathbf{C}^+\mathbf{A}\mathbf{R}^+)\mathbf{R},
\end{align}
where $\mathbf{C} = \mathbf{A}(:,J)$ is a subset of representative columns of $\mathbf{A}$, $\mathbf{R} = \mathbf{A}(I,:)$ is a subset of representative rows of $\mathbf{A}$, and the mixing matrix is $\mathbf{U}$. The properties of the approximation depend critically on the rank of $\mathbf{U}$ and the selection of columns and rows, $J$ and $I$, as demonstrated by Sorensen and Embree \cite{sorensen2016deim}, Wang and Zhang \cite{wang2013improving}, Martinsson and Tropp \cite{martinsson2019randomized}, Hamm and Huang \cite{hamm2020perspectives}. 

To obtain the general reduced form of the CUR decomposition, we can set $\mathbf{F}=\mathbf{C}$, $\mathbf{G}=\mathbf{C}^+\mathbf{A}\mathbf{R}^+$, and $\mathbf{H}^*=\mathbf{R}$. Then, if $\mathrm{rank}(\mathbf{U}) = \mathrm{rank}(\mathbf{A})$ given the selection of columns and rows, $J$ and $I$, we get the exact reconstruction:
\begin{align}
\begin{aligned}
\mathbf{A} = 
\begin{bmatrix}
0 & 1 & 1/2\\ 1 & 2 & 3/2 \\2 & 7 & 9/2
\end{bmatrix}
=
\begin{bmatrix}
0 & 1 \\ 1 & 2\\ 2 & 7
\end{bmatrix}
\begin{bmatrix}
-2 & 1 \\ \phantom{-}1 & 0
\end{bmatrix} 
\begin{bmatrix}
0 & 1 & 1/2\\ 1 & 2 & 3/2 
\end{bmatrix}
= \mathbf{C}\mathbf{U}\mathbf{R}.
\end{aligned}
\end{align}

\paragraph{Outer-product decomposition.}
Finally, consider the case of diagonal rank-$k$ mixing matrix $\mathbf{G} = \mathrm{diag}(g_1,\dots,g_k)$. The general reduced form of matrix factorization can now be written as the weighted sum of rank-$1$ matrices:
\begin{align}\label{eq:outer-product}
\mathbf{A}  =
\mathbf{F}\cdot \mathrm{diag}(g_1,\dots,g_k)\cdot\mathbf{H}^* 
=
\sum_{i=1}^k g_i \mathbf{f}_i\mathbf{h}_i^*.
\end{align}
That is the case for the SVD and CPQR decomposition, where  $\mathbf{G} = \mathbf{S}(1\colon k,1\colon k)$ and $\mathbf{G} = \mathbf{I}_k$ are diagonal matrices. That is also the case for any factorization $\mathbf{A} = \mathbf{B}\mathbf{D}$, which can be written as the sum of columns times rows:
\begin{align}
\mathbf{A}  =
\mathbf{B}\mathbf{D} 
=
\sum_{i=1}^k \mathbf{b}_i\mathbf{d}_i^*.
\end{align}
For more on the outer-product decomposition, see Wedderburn \cite{wedderburn1934lectures}, Egerv{\'a}ry \cite{egervary1960rank}, Householder \cite{householder1965theory}, Cline and Funderlic \cite{cline1979rank}.

We can learn much about matrix factorizations by studying their general reduced form. A good way to start is to understand the factors involved.

\section{The key in the middle}
\label{sec:The key in the middle}

Matrices $\mathbf{F}$ and $\mathbf{H}^*$ provide the building blocks for the reconstruction of $\mathbf{A}$. Matrix $\mathbf{F}$ is a basis for the column space of $\mathbf{A}$, because every column of $\mathbf{A}$ must be a linear combination of the columns of $\mathbf{F}$. Similarly, matrix $\mathbf{H}^*$ is a basis for the row space of $\mathbf{A}$ because every row of $\mathbf{A}$ must be a linear combination of the rows of $\mathbf{H}^*$. However, we need one more matrix to put all those blocks together. The missing factor is $\mathbf{G}$, the mixing matrix controlling the reconstruction process. It connects the columns of $\mathbf{F}$ with the rows of $\mathbf{H}^*$ and makes the reconstruction possible. 

To understand what $\mathbf{G}$ is, we will call the Moore-Penrose \cite{penrose1955generalized} pseudoinverses for help. The Moore-Penrose pseudoinverse of $\mathbf{F}$ is a unique solution to the following system of equations:
\begin{align}\label{eq:Penrose-eqs}
\begin{aligned}
\mathbf{F}\mathbf{F}^+\mathbf{F} = \mathbf{F}
\quad&\text{and}\quad
\mathbf{F}^+\mathbf{F}\mathbf{F}^+ = \mathbf{F}^+
\\
(\mathbf{F}\mathbf{F}^+)^* = \mathbf{F}\mathbf{F}^+
\quad&\text{and}\quad
(\mathbf{F}^+\mathbf{F})^* = \mathbf{F}^+\mathbf{F}.
\end{aligned}    
\end{align}
Note that the same equations hold for $(\mathbf{H}^*)^+$. 
The first two define the basic mechanics of inversion and are essential in our quest as we are about to see. They are known as generalized inverse conditions. The last two impose symmetry of matrix products. 

Suppose that 
\begin{align}
\begin{aligned}
\mathbf{F} = 
\begin{bmatrix}
1\\ 0\\1
\end{bmatrix}.
\end{aligned}
\end{align}
We can verify that:
\begin{align}
\begin{aligned}
\mathbf{F}^+=
\begin{bmatrix}
1/2 & 0 & 1/2
\end{bmatrix}
\end{aligned}
\end{align}
solves the Penrose's equations:
\begin{align}
\begin{aligned}
\mathbf{F}\mathbf{F}^+\mathbf{F} &=
\begin{bmatrix}
1\\ 0\\1
\end{bmatrix}
\begin{bmatrix}
1/2 & 0 & 1/2
\end{bmatrix}
\begin{bmatrix}
1\\ 0\\1
\end{bmatrix}=
\begin{bmatrix}
1\\ 0\\1
\end{bmatrix} = \mathbf{F},
\end{aligned}
\end{align}

\begin{align}
\begin{aligned}
\mathbf{F}^+\mathbf{F}\mathbf{F}^+ &=
\begin{bmatrix}
1/2 & 0 & 1/2
\end{bmatrix}
\begin{bmatrix}
1\\ 0\\1
\end{bmatrix}
\begin{bmatrix}
1/2 & 0 & 1/2
\end{bmatrix}
=
\begin{bmatrix}
1/2 & 0 & 1/2
\end{bmatrix} = \mathbf{F}^+,
\end{aligned}
\end{align}

\begin{align}
\begin{aligned}
(\mathbf{F}\mathbf{F}^+)^* &=
\begin{bmatrix}
1/2 & 0 & 1/2\\
0 & 0 & 0\\
1/2 & 0 & 1/2
\end{bmatrix} = \mathbf{F}\mathbf{F}^+
\quad\mathrm{and}\quad
(\mathbf{F}^+\mathbf{F})^* &=
\begin{bmatrix}
1/2 & 0 & 1/2
\end{bmatrix}
\begin{bmatrix}
1\\0\\1
\end{bmatrix} = 1 = \mathbf{F}^+\mathbf{F}.
\end{aligned}
\end{align}
Provided that all the necessary assumptions about matrices involved are made, multiplying the general reduced form by $\mathbf{F}^{+}$ and $(\mathbf{H}^*)^{+}$ results in:
\begin{align}\label{eq:Gfactorization}
\mathbf{G} = \mathbf{F}^{+} \mathbf{A} (\mathbf{H}^*)^+.
\end{align}
However, that is a shortcut we do not want to take. Instead, let us then assume that:
\begin{align}\label{eq:geniotransform}
\setstackgap{L}{15pt}\def\stacktype{L}\def\sz{\scriptstyle}
\begin{aligned}
\stackunder{\mathbf{G}}{\sz (k\times k)} &=  
\stackunder{\mathbf{Y}^*}{\sz (k\times m)}&& 
\stackunder{\mathbf{A}}{\sz (m\times n)} &&
\stackunder{\mathbf{X}.}{\sz (n\times k)}
\end{aligned}
\end{align}
We do not know what $\mathbf{Y}^*$ and $\mathbf{X}$ are, but now we can ask questions about them. Answering these questions will provide us with the desired insights into the mechanics of interactions between the factorizations' factors.

\section{Hidden projectors}
\label{sec:Hidden projectors}

Good science is about asking good questions, which is what we are about to do. We want to ask good questions about $\mathbf{Y}^*$ and $\mathbf{X}$. For that, a mighty trick would come in handy. 

Let us substitute the formula for the mixing matrix into the general reduced form of matrix factorizations. What we get as a result is the key equation of this paper. The following equation unravels what holds for every matrix factorization:
\begin{align}\label{eq:motivation}
\mathbf{A} =
(\mathbf{F} \mathbf{Y}^*) 
\mathbf{A} 
(\mathbf{X} \mathbf{H}^*).
\end{align}
To better understand that feature, see that matrix $\mathbf{A}$ is present both on the left-hand and the right-hand side of the equation. That fact makes the two matrices in the brackets, $\mathbf{F} \mathbf{Y}^*$ and $\mathbf{X} \mathbf{H}^*$, very special. For the equation to work, both matrices must behave like projectors. Every matrix factorization defines two hidden projectors, $\mathbf{F} \mathbf{Y}^*$ for the column space and $\mathbf{X} \mathbf{H}^*$ for the row space of the original matrix. They project the original matrix onto its own column space and row space. That sheds new light on how matrix factorizations work. 

The next question is, when do matrices $\mathbf{F} \mathbf{Y}^*$ and $\mathbf{X} \mathbf{H}^*$ behave like projectors? 

\section{The projection equation, its solutions and beyond}
\label{sec:The projection equation}

There is no royal way to learn how projectors work, so let us dive into the details. Consider a space in which every vector is a linear combination (i.e., a weighted sum) of the columns of $\mathbf{F}$. Suppose also that vector $\mathbf{b}$ cannot be represented by any such combination. Still, we would like to find $\mathbf{p} = \mathbf{F}\mathbf{g}$, a~linear combination of the columns of $\mathbf{F}$, that is as close to $\mathbf{b}$ as possible. 

To do that, we can minimize the length of the error vector $\mathbf{e} = \mathbf{b}-\mathbf{p}$. By doing so, we will find the vector $\mathbf{p}= \mathbf{F}\mathbf{g}$ we are looking for. The error vector has minimal length when it is perpendicular to every column of $\mathbf{F}$, which means it solves the following equation:
\begin{align}
\mathbf{F}^T\mathbf{e} = \mathbf{0}.    
\end{align}
If the columns of $\mathbf{F}$ are linearly independent, then the solution to the equation above implies that:
\begin{align}
\mathbf{g} = (\mathbf{F}^T\mathbf{F})^{-1}\mathbf{F}^T\mathbf{b}.    
\end{align}
Otherwise, if $\mathbf{F}$ contains linearly dependent columns, then the general solution is:
\begin{align}
\mathbf{g} = \mathbf{F}^+\mathbf{b}.    
\end{align}
Multiplying by $\mathbf{F}$ gives us the vector we are looking for, namely,
\begin{align}
\mathbf{p} = \mathbf{F}\mathbf{g} = \mathbf{F}\mathbf{F}^+\mathbf{b}.    
\end{align}
This particular linear combination of the columns of $\mathbf{F}$ minimizes the distance between $\mathbf{b}$ and $\mathbf{p}$. 

The matrix on the right-hand side of the equation, $\mathbf{F}\mathbf{F}^+$, is known as the projection matrix. It projects $\mathbf{b}$ onto the space spanned by the columns of $\mathbf{F}$ producing vector $\mathbf{p}$. The projection is orthogonal, which means that $\mathbf{e} =  \mathbf{b} - \mathbf{p}$ minimizes the mean square projection error.

The projection matrix has the following defining property:
\begin{align}
(\mathbf{F}\mathbf{F}^+)^2 = \mathbf{F}\mathbf{F}^+\mathbf{F}\mathbf{F}^+ =\mathbf{F}\mathbf{F}^+.
\end{align}
That should look familiar because we have already seen a similar matrix equation. The equation above holds because of the basic inversion mechanics defining the Moore-Penrose pseudoinverse. The very same property defines the projection matrix. In fact, that is the property we must exploit to show when matrices $\mathbf{F} \mathbf{Y}^*$ and $\mathbf{X} \mathbf{H}^*$ behave like projectors.

Indeed, we want them to be \textit{idempotent}, which means that the following conditions hold:
\begin{align}
(\mathbf{F}\mathbf{Y}^*)^2 = \mathbf{F}(\mathbf{Y}^*\mathbf{F})\mathbf{Y}^* =\mathbf{F}\mathbf{Y}^*
\quad\text{and}\quad
(\mathbf{X}\mathbf{H}^*)^2 = \mathbf{X}(\mathbf{H}^*\mathbf{X})\mathbf{H}^*=\mathbf{X}\mathbf{H}^*.
\end{align}
It is straightforward to see that it happens if
\begin{align}\label{eq:projector-eq}
\mathbf{Y}^*\mathbf{F} = \mathbf{I}_k = \mathbf{H}^*\mathbf{X}.
\end{align}
We can call that the \textit{projector equation}. It holds if and only if $\mathbf{F} \mathbf{Y}^*$ and $\mathbf{X} \mathbf{H}^*$ are idempotent rank-$k$ matrices, which means that they behave like (oblique) projectors (see Langenhop \cite{langenhop1967generalized}, and Karpowicz \cite{KarpowiczMFT} for more details). That is the answer to the question about the factorization mechanics we are looking for. The projector equation is the sufficient condition. Matrices $\mathbf{F} \mathbf{Y}^*$ and $\mathbf{X} \mathbf{H}^*$ behave like projectors if the projector equation holds. 

The projector equation is very useful. It tells us much about the solutions it admits. These solutions can take the following form:
\begin{align}\label{eq:YX-projector}
\mathbf{Y}^* = (\mathbf{B}^*\mathbf{F})^+\mathbf{B}^*
\quad\text{and}\quad
\mathbf{X} = \mathbf{D}(\mathbf{H}^*\mathbf{D})^+.
\end{align}
However, matrices $\mathbf{B}$ and $\mathbf{D}$ must be rank-preserving,~i.e., $\mathrm{rank}(\mathbf{B}^*\mathbf{F}) = k = \mathrm{rank}(\mathbf{H}^*\mathbf{D})$ for rank-$k$ $\mathbf{F}$ and $\mathbf{H}$. 

In the case of full-rank matrices, it is easy to see that:
\begin{align}
(\mathbf{B}^*\mathbf{F})^+\mathbf{B}^*\mathbf{F}
= \mathbf{I}_k = \mathbf{H}^*\mathbf{D}(\mathbf{H}^*\mathbf{D})^+.
\end{align}
Consider the following example again:
\begin{align}
\begin{aligned}
\mathbf{A} = 
\begin{bmatrix}
0 & 1 & 1/2\\ 1 & 2 & 3/2 \\2 & 7 & 9/2
\end{bmatrix},\quad\text{where}\quad\mathrm{rank}(\mathbf{A}) = 2.
\end{aligned}
\end{align}
Suppose that:
\begin{align}
\begin{aligned}
\mathbf{F} = 
\begin{bmatrix}
0 & 1 \\ 1 & 2 \\2 & 7
\end{bmatrix}
\quad\text{and}\quad
\mathbf{H}^* =
\begin{bmatrix}
0 & 1 & 1/2\\ 1 & 2 & 3/2 
\end{bmatrix},
\end{aligned}
\end{align}
and consider randomly selected:
\begin{align}
\begin{aligned}
\mathbf{B} = 
\begin{bmatrix}
1 & 0 \\ 0 & 1 \\1 & 1
\end{bmatrix}
\quad\text{and}\quad
\mathbf{D} =
\begin{bmatrix}
1 & 0\\0 & 1\\ 0 & 1 
\end{bmatrix}.
\end{aligned}
\end{align}
It is easy to verify that $\mathrm{rank}(\mathbf{B}^*\mathbf{F}) = \mathrm{rank}(\mathbf{H}^*\mathbf{D}) = k = 2$, so $\mathbf{B}$ and $\mathbf{D}$ are rank-preserving and yield:
\begin{align}
\begin{aligned}
\mathbf{Y}^* &= (\mathbf{B}^*\mathbf{F})^+\mathbf{B}^* =
\begin{bmatrix}
-3/2 & \phantom{-}4/3\\ \phantom{-}1/2 & -1/3
\end{bmatrix}
\begin{bmatrix}
1 & 0 & 1\\ 0 & 1 & 1
\end{bmatrix} = 
\begin{bmatrix}
-3/2 & \phantom{-}4/3 & -1/6\\\phantom{-}1/2 & -1/3 & \phantom{-}1/6
\end{bmatrix},
\\
\mathbf{X} &= \mathbf{D}(\mathbf{H}^*\mathbf{D})^+ =
\begin{bmatrix}
1 & 0 \\ 0 & 1\\ 0 & 1 
\end{bmatrix}
\begin{bmatrix}
-7/3 & 1\\ \phantom{-}2/3 & 0
\end{bmatrix}
=
\begin{bmatrix}
-7/3 & 1\\ \phantom{-}2/3 & 0 \\ \phantom{-}2/3 &0
\end{bmatrix}.
\end{aligned}
\end{align}
As a result, we have:
\begin{align}
\begin{aligned}
\mathbf{Y}^*\mathbf{F} = \mathbf{H}^*\mathbf{X} = 
\begin{bmatrix}
1 & 0\\ 0 & 1
\end{bmatrix},
\end{aligned}
\end{align}
so the projector equation holds and $\mathbf{F}\mathbf{Y}^*$ and $\mathbf{X}\mathbf{H}^*$ are idempotent matrices.

Let us now address the general case in which rank-$k$ matrices $\mathbf{F}$, $\mathbf{H}$, $\mathbf{B}^*\mathbf{F}$ and $\mathbf{H}^*\mathbf{D}$ are not necessarily full-rank.  
Can $\mathbf{F}\mathbf{Y}^*$ and $\mathbf{X}\mathbf{H}^*$ still be idempotent and behave like projectors? The following result shows that the answer is yes, if $\mathrm{rank}(\mathbf{B}^*\mathbf{F}) = \mathrm{rank}(\mathbf{H}^*\mathbf{D}) = k$. Furthermore, $\mathbf{Y}^*$ and $\mathbf{X}$ preserve the key property of generalized inverses (see also Hamm \cite{HAMM2023236} for recent results involving important special cases of matrix decompositions).

\begin{theorem}[Karpowicz \cite{KarpowiczMFT}]\label{Theorem:idempotence}
Suppose that 
\begin{align}
\mathbf{Y}^* = (\mathbf{B}^*\mathbf{F})^+\mathbf{B}^*
\quad\text{and}\quad
\mathbf{X} = \mathbf{D}(\mathbf{H}^*\mathbf{D})^+,
\end{align}
where $\mathbf{B}$ and $\mathbf{D}$ satisfy the following rank-preserving condition: 
\begin{align}
\mathrm{rank}(\mathbf{B}^*\mathbf{F}) = \mathrm{rank}(\mathbf{H}^*\mathbf{D}) = \mathrm{rank}(\mathbf{F}) =\mathrm{rank}(\mathbf{H}) = k.
\end{align}
Then $\mathbf{Y}^*$ and $\mathbf{X}$ are generalized inverses,~i.e.,
\begin{align}
\mathbf{F}\mathbf{Y}^*\mathbf{F} = \mathbf{F}
\quad\text{and}\quad
\mathbf{H}^*\mathbf{X}\mathbf{H}^* = \mathbf{H}^*.
\end{align}
\end{theorem}

We can illustrate that with the following example. Consider:
\begin{align}
\begin{aligned}
\mathbf{F} = \mathbf{H}^* = 
\begin{bmatrix}
0 & 1 & 1/2\\ 0 & 2 & 1
\end{bmatrix}
\end{aligned}
\end{align}
and:
\begin{align}
\begin{aligned}
\mathbf{B} = 
\begin{bmatrix}
1 & 0 & 1\\ 2 & 0 & 2
\end{bmatrix}
\quad\text{and}\quad
\mathbf{D} =
\begin{bmatrix}
1 & 1\\0 & 0\\ 1 & 1
\end{bmatrix}.
\end{aligned}
\end{align}
We have $\mathrm{rank}(\mathbf{F}) =\mathrm{rank}(\mathbf{H}^*)=\mathrm{rank}(\mathbf{B}^*\mathbf{F}) = \mathrm{rank}(\mathbf{H}^*\mathbf{D}) = k = 1$ and the matrices are rank-deficient. However:
\begin{align}
\begin{aligned}
\mathbf{Y}^* &= (\mathbf{B}^*\mathbf{F})^+\mathbf{B}^* =
\begin{bmatrix}
0    & 0 & 0       \\
2/25 & 0 & 2/25    \\
1/25 & 0 & 1/25   
\end{bmatrix}
\begin{bmatrix}
1 & 2 \\
0 & 0 \\    
1 & 2    
\end{bmatrix} = 
\begin{bmatrix}
0    &  0    \\   
4/25 &  8/25 \\   
2/25 &  4/25
\end{bmatrix}
\end{aligned}
\end{align}
and
\begin{align}
\begin{aligned}
\mathbf{X} &= \mathbf{D}(\mathbf{H}^*\mathbf{D})^+ =
\begin{bmatrix}
1 & 1 \\       
0 & 0 \\      
1 & 1 
\end{bmatrix}
\begin{bmatrix}
1/5 & 2/5 \\     
1/5 & 2/5  
\end{bmatrix}
=
\begin{bmatrix}
2/5 & 4/5 \\     
0   & 0   \\  
2/5 & 4/5 
\end{bmatrix}
\end{aligned}
\end{align}
are generalized inverses. Indeed, we have
\begin{align}
\begin{aligned}
\mathbf{Y}^*\mathbf{F} =  
\begin{bmatrix}
0 & 0   &  0   \\  
0 & 4/5 &  2/5 \\  
0 & 2/5 &  1/5 
\end{bmatrix}
\quad\mathrm{and}\quad
\mathbf{H}^*\mathbf{X} =
\begin{bmatrix}
1/5 & 2/5 \\     
2/5 & 4/5    
\end{bmatrix},
\end{aligned}
\end{align}
and, as it is easy to calculate, 
\begin{align}
\begin{aligned}
\mathbf{F}\mathbf{Y}^*\mathbf{F}
&=
\begin{bmatrix}
0 & 1 & 1/2\\ 0 & 2 & 1
\end{bmatrix}
\begin{bmatrix}
0 & 0   &  0   \\  
0 & 4/5 &  2/5 \\  
0 & 2/5 &  1/5 
\end{bmatrix} = 
\begin{bmatrix}
0 & 1 & 1/2\\ 0 & 2 & 1
\end{bmatrix}
=
\mathbf{F},
\\
\mathbf{H}^*\mathbf{X}\mathbf{H}^*
&=
\begin{bmatrix}
1/5 & 2/5 \\     
2/5 & 4/5    
\end{bmatrix}
\begin{bmatrix}
0 & 1 & 1/2\\ 0 & 2 & 1
\end{bmatrix}
=
\begin{bmatrix}
0 & 1 & 1/2\\ 0 & 2 & 1
\end{bmatrix}
=
\mathbf{H}^*.
\end{aligned}
\end{align}
Also, the idempotence condition holds,
\begin{align}
\begin{aligned}
\mathbf{F}(\mathbf{Y}^*\mathbf{F})\mathbf{Y}^*=\mathbf{F}\mathbf{Y}^*
\quad\mathrm{and}\quad
\mathbf{X}(\mathbf{H}^*\mathbf{X})\mathbf{H}^*=\mathbf{X}\mathbf{H}^*,
\end{aligned}
\end{align}
as required.

Let us summarize what we have established so far. 
Matrix factorizations of the form
\begin{align}
\mathbf{A} = \mathbf{F}\mathbf{G}\mathbf{H}^*
\end{align}
use factors $\mathbf{F}$ and $\mathbf{H}^*$ to define two projectors, one for the column space and one for the row space of $\mathbf{A}$. For rank-preserving $\mathbf{B}$ and $\mathbf{D}$, the projectors can take the form of: 
\begin{align}
\mathbf{F}(\mathbf{B}^*\mathbf{F})^+\mathbf{B}^*
\quad\mathrm{and}\quad
\mathbf{D}(\mathbf{H}^*\mathbf{D})^+\mathbf{H}^*.
\end{align}
The projectors are not orthogonal, they are oblique and depend on parameters that can be selected subject to the rank-preserving conditions.

\section{Matrix decomposition as a solution of linear matrix equation}
\label{sec:Matrix decompositions}

The insights above suggest a relationship between the concept of matrix factorization and linear matrix equation. We want to understand this relationship fully.

Consider the following linear equation:
\begin{align}
\mathbf{F}\mathbf{g} = \mathbf{a}.
\end{align}
It is satisfied when vector $\mathbf{a}$ is a linear combination of the columns of $\mathbf{F}$. Otherwise, if $\mathbf{a}$ were unreachable for any linear combination of the columns of $\mathbf{F}$, we would face the projection problem addressed earlier. But what if we have been facing the projection problem all the time? Let us take a closer look at the well-known general solution to the system of linear equations. 

The general solution is a sum of a particular solution and a solution to the homogeneous equation, $\mathbf{F}\mathbf{g} = \mathbf{0}$. Given an arbitrary vector $\mathbf{y}$,  we can design it as: 
\begin{align}
\mathbf{g} = \mathbf{F}^+\mathbf{a} + 
(\mathbf{I} - \mathbf{F}^+\mathbf{F})\mathbf{y}.
\end{align}
Indeed, multiplying by $\mathbf{F}$ gives us the vector we are looking for,
\begin{align}
\mathbf{F}\mathbf{g} = \mathbf{F}\mathbf{F}^+\mathbf{a} + (\mathbf{F}-\mathbf{F}\mathbf{F}^+\mathbf{F})\mathbf{y}
= \mathbf{F}\mathbf{F}^+\mathbf{a} = \mathbf{a}.
\end{align}
Notice that we have seen all those components before! The general solution projects the right-hand side vector $\mathbf{a}$ onto the column space of $\mathbf{F}$. It works because of the properties of the projection matrix. We have designed the solution vector $\mathbf{g}$ with the factors of projections reconstructing $\mathbf{a}$ from the columns of $\mathbf{F}$. It takes the projection matrix $\mathbf{F}\mathbf{F}^+$ to reconstruct $\mathbf{a}$ in the column space to which it belongs. The same property of the projection matrix makes the solution to a homogeneous equation, $(\mathbf{I} - \mathbf{F}^+\mathbf{F})\mathbf{y}$, disappear when multiplied by~$\mathbf{F}$. 

Solving linear equations is about projecting vectors onto the right subspaces. Matrix factorization is about projecting columns and rows onto the fundamental subspaces given by the original matrix. The first one is the special case of the latter. That is the relationship between linear equations and matrix factorizations.

We can now connect the dots and go one step further. Consider:
\begin{align}
\mathbf{g} = \mathbf{Y}^*\mathbf{a} + (\mathbf{I}-\mathbf{Y}^*\mathbf{F})\mathbf{w}.
\end{align}
If $\mathbf{F}\mathbf{Y}^*$ is the column space projector for $\mathbf{A}$, then multiplying $\mathbf{g}$ by $\mathbf{F}$ we get
\begin{align}
\mathbf{F}\mathbf{g} = \mathbf{F}\mathbf{Y}^*\mathbf{a} + (\mathbf{F}-\mathbf{F}\mathbf{Y}^*\mathbf{F})\mathbf{w}
= \mathbf{F}\mathbf{Y}^*\mathbf{a} = \mathbf{a}.
\end{align}
That happens when the projector equation holds, or $\mathbf{Y}^*$ is generalized inverse. 

That brings us to the conclusion and the theorem below. Matrix factorizations are solutions to linear matrix equations. The proof is constructive and follows the reasoning presented by Penrose in \cite{penrose1955generalized} while applying the properties of generalized inverses given by Theorem~\ref{Theorem:idempotence} (see also see Ben-Israel and Greville \cite{ben2003generalized}).  

\begin{theorem}[Penrose \cite{penrose1955generalized}]
\label{thm:Penrose}
A~necessary and sufficient condition for the equation
\begin{align}\label{eq:PenroseLME-oblique}
\mathbf{F}\mathbf{G}\mathbf{H}^* = \mathbf{A}
\end{align}
to have a solution $\mathbf{G}$ is
\begin{align}\label{eq:PenroseNSC-oblique}
\mathbf{F}\mathbf{Y}^*
\mathbf{A}
\mathbf{X}\mathbf{H}^* = \mathbf{A},
\end{align}
where $\mathbf{Y}^*$ and $\mathbf{X}$ are generalized inverses. The general solution is
\begin{align}\label{eq:PenroseProj-oblique}
\mathbf{G} = 
\mathbf{Y}^*
\mathbf{A}
\mathbf{X}
+ 
\mathbf{W} - \mathbf{Y}^*\mathbf{F}\mathbf{W}\mathbf{H}^*\mathbf{X},
\end{align}
where $\mathbf{W}$ is arbitrary.
\end{theorem}

\begin{proof}
If $\mathbf{G}$ satisfies (\ref{eq:PenroseLME-oblique}), then $\mathcal{C}(\mathbf{F}) = \mathcal{C}(\mathbf{A})$ and $\mathcal{C}(\mathbf{H}^*) = \mathcal{C}(\mathbf{A}^*)$, and given idempotent $\mathbf{F}\mathbf{Y}^*$ and $\mathbf{H}^*\mathbf{X}$, we have:
\begin{align}
\mathbf{A} &= 
\mathbf{F}
\mathbf{G}
\mathbf{H}^*
=
\mathbf{F}\mathbf{Y}^*\mathbf{F}
\mathbf{G}
\mathbf{H}^*\mathbf{X}\mathbf{H}^*
=
\mathbf{F}\mathbf{Y}^*
\mathbf{A}
\mathbf{X}\mathbf{H}^*.
\end{align}
Therefore, we can construct the following solution of (\ref{eq:PenroseLME-oblique}):
\begin{align}
\mathbf{G} = 
\mathbf{Y}^*
\mathbf{A}
\mathbf{X}.
\end{align}

Conversely, if (\ref{eq:PenroseNSC-oblique}) holds, then $\mathbf{G} =\mathbf{Y}^*\mathbf{A}\mathbf{X}$ is a particular solution of (\ref{eq:PenroseLME-oblique}). To get the general solution, we must introduce a solution to the homogeneous equation, $\mathbf{F}\mathbf{G}\mathbf{H}^* = \mathbf{0}$. Consider
\begin{align}
\mathbf{G} = \mathbf{W} - \mathbf{Y}^*\mathbf{F}\mathbf{W}\mathbf{H}^*\mathbf{X}.
\end{align}
It is now easy to see that the equation holds if $\mathbf{Y}^*$ and $\mathbf{X}$ are the generalized inverses, which indeed is the case when (\ref{eq:YandX}) hold as demonstrated in \cite{KarpowiczMFT}. It follows that:
\begin{align}
\begin{aligned}
\mathbf{F}\mathbf{G}\mathbf{H}^* 
&= 
\mathbf{F}\mathbf{W}\mathbf{H}^* - \mathbf{F}\mathbf{Y}^*\mathbf{F}\mathbf{W}\mathbf{H}^*\mathbf{X}\mathbf{H}^*
= \mathbf{0}.
\end{aligned}
\end{align}
Therefore, the general solution to (\ref{eq:PenroseLME-oblique}) is given by (\ref{eq:PenroseProj-oblique}).
\end{proof}

The next step is to introduce hidden projectors emerging from factorization factors. We will now formulate our main result. It extends the reduced form of matrix factorization to the general form admitting rank-deficient matrices and shows how to design matrix decomposition with hidden projectors.

\begin{theorem}\label{thm:MFT-LME}
Consider a rank-$k$ matrix $\mathbf{A}\in\mathbb{K}^{m\times n}$.
Suppose $\mathbf{F}\in\mathbb{K}^{m\times r}$ and $\mathbf{H}\in\mathbb{K}^{m\times q}$, where $k\le r$ and $k\le q$, are such that $\mathcal{C}(\mathbf{A}) \subset \mathcal{C}(\mathbf{F})$ and $\mathcal{C}(\mathbf{A}^*) \subset \mathcal{C}(\mathbf{H})$. Also, suppose that $\mathbf{B}\in\mathbb{K}^{m\times r}$ and $\mathbf{D}\in\mathbb{K}^{n\times q}$ meet the following rank-preserving condition:
\begin{align}
\mathrm{rank}(\mathbf{B}^*\mathbf{F}) = \mathrm{rank}(\mathbf{H}^*\mathbf{D}) = k.
\end{align}

If 
\begin{align}\label{eq:YandX}
\mathbf{Y}^* = (\mathbf{B}^*\mathbf{F})^+\mathbf{B}^*
\quad\text{and}\quad
\mathbf{X} = \mathbf{D}(\mathbf{H}^*\mathbf{D})^+,
\end{align}
then $\mathbf{A}$ can be factorized in the form
\begin{align}\label{eq:MF-LME}
\mathbf{A} = \mathbf{F}\mathbf{G}\mathbf{H}^*,
\end{align}
where
\begin{align}\label{eq:MF-LME-solution}
\mathbf{G} = 
\mathbf{Y}^*
\mathbf{A}
\mathbf{X}
+ 
\mathbf{W} - \mathbf{Y}^*\mathbf{F}\mathbf{W}\mathbf{H}^*\mathbf{X}
\end{align}
and $\mathbf{W}$ is arbitrary.
\end{theorem}

\begin{proof}
By assumption, each column of $\mathbf{A}$ can be represented by a combination of columns of $\mathbf{F}$ and each row of $\mathbf{A}$ by a combination of rows of $\mathbf{H}$.
By Theorem \ref{Theorem:idempotence}, $\mathbf{Y}^*$ and $\mathbf{X}$ are generalized inverses. It follows that $\mathbf{F}\mathbf{Y}^*$ and $\mathbf{X}\mathbf{H}^*$ are projectors for $\mathcal{C}(\mathbf{A})$ and $\mathcal{C}(\mathbf{A}^*)$ and, by Theorem~\ref{thm:Penrose}, $\mathbf{G}$ solves $\mathbf{A} = \mathbf{F}\mathbf{G}\mathbf{H}^*$. Since
$\mathbf{F}\mathbf{W}\mathbf{H}^* = \mathbf{F}\mathbf{Y}^*\mathbf{F}\mathbf{W}\mathbf{H}^*\mathbf{X}\mathbf{H}^*$ for any $\mathbf{W}$, by the properties of projections we have:
\begin{align}
\mathbf{F}\mathbf{G}\mathbf{H}^*
= 
\mathbf{F}\mathbf{Y}^*\mathbf{A}\mathbf{X}\mathbf{H}^*
=
\mathbf{A},
\end{align}
as desired.
\end{proof}

The theorem shows how to reconstruct a rank-k matrix from a product of matrices of possibly higher rank that are possibly rank-deficient. The representation depends on parameters that can be selected subject to rank-preserving conditions. Notice that we have extended the reduced model to the general one so that:
\begin{align}\label{eq:GMMF}
\setstackgap{L}{15pt}\def\stacktype{L}\def\sz{\scriptstyle}
\begin{aligned}
\stackunder{\mathbf{A}}{\sz (m\times n)} &= 
\stackunder{\mathbf{F}}{\sz (m\times r)} &&
\stackunder{\mathbf{G}}{\sz (r\times q)} &&
\stackunder{\mathbf{H}^*.}{\sz (q\times n)}
\end{aligned}
\end{align}

The mixing matrix $\mathbf{G}$ is now rectangular, but it is still defined by the factors $\mathbf{F}$ and $\mathbf{H}^*$. These, in turn, can be derived in many ways, some of which we saw in Section~\ref{sec:One form to rule them all}. The following Matlab code illustrates the key concepts.

\begin{verbatim}
% A is given, define settings
[m,n]   = size(A);
k       = rank(A);
r       = ...; % e.g., r = ceil(k*3);
q       = ...; % e.g., q = ceil(k*2);
% Define factors of projectors (random sampling of A)
F       = A*randn(n,r);     B = rand(m,r);
H       = A'*randn(m,q);    D = rand(n,q);
% Check rank conditions
rankDef = [k,rank(B'*F),rank(F),rank(H'*D), rank(H)]
% Define generalized inverses
Y       = (pinv(B'*F)*B')'; X = D*pinv(H'*D);
% Calculate G and reconstruct A
W       = randn(r,q);
G       = Y'*A*X + W - Y'*F*W*H'*X; 
Ar      = F*G*H';
% Calculate errors
err     = norm(Ar-A,'fro')/norm(A,'fro')
err1    = norm(F*Y'*F-F,'fro')
err2    = norm(H'*X*H'-H','fro')
\end{verbatim}

\section{Making it all work}
\label{sec:Making it all work}

Let us see how $\mathbf{F}$ and $\mathbf{H}^*$ define a solution $\mathbf{G}$ to a linear matrix equation. Consider the case of the CUR decomposition of:
\begin{align}
\begin{aligned}
\mathbf{A} = 
\begin{bmatrix}
0 & 1 & 1/2\\ 1 & 2 & 3/2 \\2 & 7 & 9/2
\end{bmatrix}.
\end{aligned}
\end{align}
Selecting $\mathbf{F}$ and $\mathbf{H}$ as before, we need to find $\mathbf{G}$ such that:
\begin{align}
\begin{aligned}
\mathbf{F}\mathbf{G}\mathbf{H}^* =  
\begin{bmatrix}
0 & 1 \\ 1 & 2 \\2 & 7
\end{bmatrix}
\mathbf{G}
\begin{bmatrix}
0 & 1 & 1/2\\ 1 & 2 & 3/2 
\end{bmatrix}
=
\begin{bmatrix}
0 & 1 & 1/2\\ 1 & 2 & 3/2 \\2 & 7 & 9/2
\end{bmatrix}
= \mathbf{A}.
\end{aligned}
\end{align}
Given the previously constructed solutions to the projector equation (\ref{eq:projector-eq}), obtained for randomly selected $\mathbf{B}$ and~$\mathbf{D}$, namely:
\begin{align}
\begin{aligned}
\mathbf{Y}^* = (\mathbf{B}^*\mathbf{F})^+\mathbf{B}^* =
\begin{bmatrix}
-3/2 & \phantom{-}4/3 & -1/6\\\phantom{-}1/2 & -1/3 & \phantom{-}1/6
\end{bmatrix}
\quad\text{and}\quad
\mathbf{X} = \mathbf{D}(\mathbf{H}^*\mathbf{D})^+ =
\begin{bmatrix}
-7/3 & 1\\ \phantom{-}2/3 & 0 \\ \phantom{-}2/3 &0
\end{bmatrix},
\end{aligned}
\end{align}
for any $\mathbf{W}$ we get:
\begin{align}
\begin{aligned}
\mathbf{G} = 
\mathbf{Y}^*
\mathbf{A}
\mathbf{X}
+ 
\mathbf{W} - \mathbf{Y}^*\mathbf{F}\mathbf{W}\mathbf{H}^*\mathbf{X}
=
\begin{bmatrix}
-2 & 1\\ \phantom{-}1 & 0
\end{bmatrix}
+ \mathbf{W} - 
\begin{bmatrix}
1 & 0\\ 0 & 1
\end{bmatrix}
\mathbf{W}
\begin{bmatrix}
1 & 0\\ 0 & 1
\end{bmatrix} = 
\begin{bmatrix}
-2 & 1\\ \phantom{-}1 & 0
\end{bmatrix}.
\end{aligned}
\end{align}
As a result:
\begin{align}
\begin{aligned}
\mathbf{F}\mathbf{G}\mathbf{H}^* =  
\begin{bmatrix}
0 & 1 \\ 1 & 2 \\2 & 7
\end{bmatrix}
\begin{bmatrix}
-2 & 1\\ \phantom{-}1 & 0
\end{bmatrix}
\begin{bmatrix}
0 & 1 & 1/2\\ 1 & 2 & 3/2 
\end{bmatrix}
=
\begin{bmatrix}
0 & 1 & 1/2\\ 1 & 2 & 3/2 \\2 & 7 & 9/2
\end{bmatrix}
= \mathbf{A}.
\end{aligned}
\end{align}

Next, consider the case of rank-deficient matrices:
\begin{align}
\begin{aligned}
\mathbf{A} = \mathbf{F} = \mathbf{H}^* = 
\begin{bmatrix}
0 & 1 & 1/2\\ 0 & 2 & 1
\end{bmatrix}
\end{aligned}
\end{align}
and:
\begin{align}
\begin{aligned}
\mathbf{B} = 
\begin{bmatrix}
1 & 0 & 1\\ 2 & 0 & 2
\end{bmatrix}
\quad\text{and}\quad
\mathbf{D} =
\begin{bmatrix}
1 & 1\\0 & 0\\ 1 & 1
\end{bmatrix}.
\end{aligned}
\end{align}
We have already found that:
\begin{align}
\begin{aligned}
\mathbf{Y}^* &= (\mathbf{B}^*\mathbf{F})^+\mathbf{B}^* =
\begin{bmatrix}
0    & 0 & 0       \\
2/25 & 0 & 2/25    \\
1/25 & 0 & 1/25   
\end{bmatrix}
\begin{bmatrix}
1 & 2 \\
0 & 0 \\    
1 & 2    
\end{bmatrix} = 
\begin{bmatrix}
0    &  0    \\   
4/25 &  8/25 \\   
2/25 &  4/25
\end{bmatrix}
\end{aligned}
\end{align}
and
\begin{align}
\begin{aligned}
\mathbf{X} &= \mathbf{D}(\mathbf{H}^*\mathbf{D})^+ =
\begin{bmatrix}
1 & 1 \\       
0 & 0 \\      
1 & 1 
\end{bmatrix}
\begin{bmatrix}
1/5 & 2/5 \\     
1/5 & 2/5  
\end{bmatrix}
=
\begin{bmatrix}
2/5 & 4/5 \\     
0   & 0   \\  
2/5 & 4/5 
\end{bmatrix}.
\end{aligned}
\end{align}

So, we have:
\begin{align}
\begin{aligned}
\mathbf{G} &= 
\begin{bmatrix}
0    &  0    \\   
4/25 &  8/25 \\   
2/25 &  4/25
\end{bmatrix}
\begin{bmatrix}
0 & 1 & 1/2\\ 0 & 2 & 1
\end{bmatrix}
\begin{bmatrix}
2/5 & 4/5 \\     
0   & 0   \\  
2/5 & 4/5 
\end{bmatrix}
+ 
\mathbf{W} - \mathbf{Y}^*\mathbf{F}\mathbf{W}\mathbf{H}^*\mathbf{X}
\\
&=
\begin{bmatrix}
0    &  0    \\   
4/25 &  8/25 \\   
2/25 &  4/25
\end{bmatrix}
+ 
\mathbf{W} - \mathbf{Y}^*\mathbf{F}\mathbf{W}\mathbf{H}^*\mathbf{X}
.
\end{aligned}
\end{align}
For any $\mathbf{W}$ we then obtain:
\begin{align}
\begin{aligned}
\mathbf{F}\mathbf{G}\mathbf{H}^* =  
\begin{bmatrix}
0 & 1 & 1/2\\ 0 & 2 & 1
\end{bmatrix}
\begin{bmatrix}
0    &  0    \\   
4/25 &  8/25 \\   
2/25 &  4/25
\end{bmatrix}
\begin{bmatrix}
0 & 1 & 1/2\\ 0 & 2 & 1
\end{bmatrix}
=
\begin{bmatrix}
0 & 1 & 1/2\\ 0 & 2 & 1
\end{bmatrix}
= \mathbf{A}.
\end{aligned}
\end{align}

\section{Hiding secrets with hidden projectors}
\label{sec:Hiding secrets with hidden projectors}

We have taken many steps to show how matrix factorization works and explain its internal mechanisms in terms of projections. It is time to apply that understanding in practice. One application of great importance comes from the world of cryptography. We aim to design a simple cryptographic system encrypting and decrypting messages. 

% Schneier and co-authors present a wonderful introduction to applied cryptography in  \cite{schneier2007applied,ferguson2011cryptography}, Menezes, Van Oorschot, and Vanstone give a detailed overview of mathematics involved in \cite{menezes2018handbook}, whereas Singh \cite{singh2000code} and Holden \cite{holden2017secrets} present an excellent historical perspective.

Let $M$ denote a message and $C$ its encrypted form, called ciphertext. The encryption function, $E$, operates on $M$ to produce~$C$:
\begin{align}
E(M) = C.
\end{align}
To decrypt $C$, we need the decryption function, $D$. It should operate on the ciphertext to retrieve the message:
\begin{align}
D(E(M)) = M.
\end{align}
One of the secure ways to implement cryptographic functions is to use a public-key encryption setting. In such a setting, each sender needs a pair of interrelated keys, a secret (or private) one, $S$, and a (commonly known) public one, $P$. The public key is necessary to encrypt a message addressed to the owner of the secret key:
\begin{align}
E(P,M) = C.
\end{align}
To decrypt the ciphertext, one needs the secret key:
\begin{align}
D(S,C) = M.
\end{align}
By design (or assumption), the only way to decrypt the message encrypted with the public key is to use the secret (private) key, so:
\begin{align}
D(S,E(P,M)) = M.
\end{align}
There is one critical rule to follow, known as the
\textit{Kerckhoffs' principle}. The encryption scheme must depend only on the secrecy of the (secret) key and not on the secrecy of the algorithm. It must be extremely difficult (practically impossible) to recover the secret key from the public key and the commonly known properties of the applied cryptographic algorithm. Let us be clear: many other challenges are involved in applied cybersecurity engineering, including the fundamental key distribution problem. In this short note, to illustrate the main concepts and pose challenging questions, we will only address Kerckhoffs' principle.

To design a public-key-based cryptographic system $(E,D,S,P)$, we will use Theorem~\ref{Theorem:idempotence} and exploit the fact that there are many different ways of constructing projectors and matrix decompositions.

We must start with a dictionary to represent messages (letters, words, sentences, etc.). It will have a form of an $m$ by $n$ rank-$k$ (rank-deficient) matrix $\mathbf{A}$. One interesting idea is to use word embeddings produced by a neural network layer trained to reconstruct the set of messages. Another is to use any numerical encoding of an alphabet or dataset.

Now we can construct an $m$ by $r$ matrix $\mathbf{F}$, where $r$ is greater than $k$, the key component of the design. Considered the case in which $\mathcal{C}(\mathbf{A})$ is the same as $\mathcal{C}(\mathbf{F})$. Notice that $\mathbf{F}$ is rank-deficient. To construct a basis for $\mathcal{C}(\mathbf{F})$, we can use any factorization algorithm. One approach is to use QR factorization. First, we create an orthogonal basis $\mathbf{Q}$ for $\mathcal{C}(\mathbf{A})$. Then, we construct $r$ columns of $\mathbf{F}$ by taking $k$ columns of $\mathbf{Q}$ and then $r-k$ combinations of those columns. Other factorization methods, such as LU, SVD, and similarity transformation, can be used similarly. We can also use random sampling (e.g., sparse binary) to construct a randomized basis, which is probably the best idea to follow since it makes $\mathcal{C}(\mathbf{A})$ difficult to guess.

There is one more step to take. Let $\mathbf{B}_1$ and $\mathbf{B}_2$ be random $m$ by $r$ matrices such that $\mathrm{rank}(\mathbf{B}_1^*\mathbf{F}) = \mathrm{rank}(\mathbf{B}_1^*\mathbf{F}) = k$. Since they are rank-preserving, 
\begin{align}
\mathbf{Y}^*_{\mathbf{B}_i} = (\mathbf{B}_i^*\mathbf{F})^+\mathbf{B}_i^*,\ i=1,2,
\end{align}
is (randomized) generalized inverse of $\mathbf{F}$. 

Consider now the following pair of keys:
\begin{align}\label{eq:CSP-MF-keys}
S = \mathbf{F}
\quad\mathrm{and}\quad
P =(\mathbf{Y}^*_{\mathbf{B}_1}, \mathbf{Y}^*_{\mathbf{B}_2}\mathbf{F}).
\end{align}
Given the public key $P$ above, the encryption function can be defined as follows:
\begin{align}\label{eq:CSP-MF-enc}
E(P,\mathbf{m}) = \mathbf{Y}^*_{\mathbf{B}_1}\mathbf{m} + (\mathbf{I}-\mathbf{Y}^*_{\mathbf{B}_2}\mathbf{F})\mathbf{w},
\end{align}
where $\mathbf{w}$ is arbitrary. The encryption function produces ciphertext $\mathbf{c}$ for any message $\mathbf{m}$ being a column of $\mathbf{A}$. 

Given the secret (private) key $S$, the associated decryption function, decrypting ciphertext $\mathbf{c}$, is then given by:
\begin{align}\label{eq:CSP-MF-dec}
D(S,\mathbf{c}) = \mathbf{F}\mathbf{c}.
\end{align}
Theorem~\ref{thm:MFT-LME} shows that the decryption reconstructs the encrypted message:
\begin{align}
D(S,E(P,\mathbf{m})) = 
\mathbf{F}\mathbf{Y}^*_{\mathbf{B}_1}\mathbf{m} + (\mathbf{F}-\mathbf{F}\mathbf{Y}^*_{\mathbf{B}_2}\mathbf{F})\mathbf{w}
= \mathbf{m}.
\end{align}
We can also produce the encrypted dictionary:
\begin{align}
E(P,\mathbf{A}) = 
\mathbf{Y}^*_{\mathbf{B}_1}\mathbf{m} + (\mathbf{I}-\mathbf{Y}^*_{\mathbf{B}_2}\mathbf{F})\mathbf{W} = \mathbf{C}.
\end{align}
We use the ones-sided random and oblique column-space projector for decryption:
\begin{align}
D(S,E(P,\mathbf{A})) = 
\mathbf{F}\mathbf{Y}^*_{\mathbf{B}_1}\mathbf{A} + (\mathbf{F}-\mathbf{F}\mathbf{Y}^*_{\mathbf{B}_2}\mathbf{F})\mathbf{W}
= \mathbf{A}.
\end{align}
Below is the Matlab code illustrating the design concept.

\begin{verbatim}
% Private key: random sketching of A
F   = A*randn(n,r);
% Public key
B1  = randn(m,r); for i = k+1:r, B1(:,i) = B1(:,1:k)*randn(k,1);end
Y1  = (pinv(B1'*F)*B1')';
B2  = randn(m,r); for i = k+1:r, B2(:,i) = B2(:,1:k)*randn(k,1);end
Y2  = (pinv(B2'*F)*B2')';
% Encryption function (introducing random noise)
E   = @(M)Y1'*M + (eye(r)-Y2'*F)*randn(r,1);
% Decryption function
D   = @(C)F*C;
% Message: random sequence Mids of N columns of A
Mids = randi([1 n],1,N);
M   = A(:,Mids);
% Encrypt M
C   = E(M)
% Decrypt C
Md  = D(C);
% Decryption error
decryptionError = norm(M-Md,'fro')
\end{verbatim}

Let us now briefly address Kerckhoffs' principle by asking if it is difficult to reconstruct $S$ from $(E,P)$? Or, in other words, given $\mathbf{Y}^*_{\mathbf{B}_1}$ and $\mathbf{Y}^*_{\mathbf{B}_i}\mathbf{F}$, is it difficult to find $\mathbf{F}$?

The first property to notice is that the encryption function depends on $\mathbf{w}$, a random variable. Since $\mathbf{w}$ can change with every execution of the encryption function, we get random ciphertexts for the same message every time we use the encryption function above. If $\mathbf{c}_t = E(P,\mathbf{m})$ is a ciphertext generated for $\mathbf{m}$ at time instant $t$ with $\mathbf{w} = \mathbf{w}_t$, then $\mathbf{c}_{k} = E(P,\mathbf{m}) \neq \mathbf{c}_t$ at time instant $k \neq t$ whenever $\mathbf{w}_k \neq \mathbf{w}_t$. We can hide the original probability distribution of messages under the probability distribution of noise introduced into the ciphertext, thereby making the attacks based on frequency analysis ineffective.

The second important property of the system is that the general inverse $\mathbf{Y}^*_{\mathbf{B}_i}$ is not unique by design. It depends on an arbitrary and random rank-preserving matrix $\mathbf{B}_i$ and satisfies only the first of the four Penrose's pseudoinverse equations (we need all four to make pseudoinverse unique). Matsaglia and Styan demonstrate that important fact in~\cite{matsaglia1974equalities}. Since any rank-$k$ matrix $\mathbf{F}$ may be written as
\begin{align}
\mathbf{F} =
\mathbf{P}
\begin{bmatrix}
\mathbf{I}_k & \mathbf{0}
\\
\mathbf{0} & \mathbf{0}
\end{bmatrix}
\mathbf{S},
\end{align}
where $\mathbf{S}$ and $\mathbf{P}$ are invertible, every solution $\mathbf{Y}^*_{\mathbf{B}_i}$ to $\mathbf{F}\mathbf{Y}^*_{\mathbf{B}_i}\mathbf{F} = \mathbf{F}$ can be written as
\begin{align}
\mathbf{Y}^*_{\mathbf{B}_i} =
\mathbf{S}^{-1}
\begin{bmatrix}
\mathbf{I}_k & \mathbf{Z}_1
\\
\mathbf{Z}_2 & \mathbf{Z}_3
\end{bmatrix}
\mathbf{P}^{-1},
\end{align}
where $\mathbf{Z}_i,\ i = 1,2,3,$ is arbitrary (random). That makes $\mathbf{Y}^*_{\mathbf{B}_i}$ non-unique. We can now apply the same argument to show that there are many solutions $\bar{\mathbf{F}}$ to $\mathbf{Y}^*_{\mathbf{B}_i}\bar{\mathbf{F}}\mathbf{Y}^*_{\mathbf{B}_i} = \mathbf{Y}^*_{\mathbf{B}_i}$. Given $\mathbf{Y}^*_{\mathbf{B}_i}$, we can write it as
\begin{align}
\mathbf{Y}^*_{\mathbf{B}_i} =
\mathbf{P}_{\mathbf{B}_i}
\begin{bmatrix}
\mathbf{I}_k & \mathbf{0}
\\
\mathbf{0} & \mathbf{0}
\end{bmatrix}
\mathbf{S}_{\mathbf{B}_i},
\end{align}
and then show that
\begin{align}
\bar{\mathbf{F}} = 
\mathbf{S}^{-1}_{\mathbf{B}_i}
\begin{bmatrix}
\mathbf{I}_k & \bar{\mathbf{Z}}_1
\\
\bar{\mathbf{Z}}_2 & \bar{\mathbf{Z}}_3
\end{bmatrix}
\mathbf{P}^{-1}_{\mathbf{B}_i}
\end{align}
is the desired solution for arbitrary $\bar{\mathbf{Z}}_i,\ i = 1,2,3.$ Clearly, one of the solutions is $(\mathbf{Y}^*_{\mathbf{B}_i})^+$, the Penrose-Moore pseudoinverse. However, since $\mathbf{F}$ is rank-deficient, $(\mathbf{Y}^*_{\mathbf{B}_i})^+\mathbf{Y}^*_{\mathbf{B}_j}$ is not an identity matrix:
\begin{align}
\mathbf{Y}^*_{\mathbf{B}_i}\mathbf{F} \neq \mathbf{I}
\quad\mathrm{and}\quad
(\mathbf{Y}^*_{\mathbf{B}_i})^+\mathbf{Y}^*_{\mathbf{B}_j}\neq \mathbf{I}.
\end{align}
It follows that 
\begin{align}
(\mathbf{Y}^*_{\mathbf{B}_1})^+\mathbf{Y}^*_{\mathbf{B}_2}\mathbf{F}\neq \mathbf{F},
\end{align}
so we cannot isolate $\mathbf{F}$ from $\mathbf{Y}^*_{\mathbf{B}_2}\mathbf{F}$ with $(\mathbf{Y}^*_{\mathbf{B}_1})^+$. Similarly, there are many solutions $\mathbf{S}$ to homogenous equation:
\begin{align}
\mathbf{S}-\mathbf{S}\mathbf{Y}^*_{\mathbf{B}_2}\mathbf{F} =  \mathbf{0}.
\end{align}
Each one of them defines a nontrivial basis for the nullspace of $\mathbf{I}-\mathbf{Y}^*_{\mathbf{B}_2}\mathbf{F}$. 

The arguments above show that it may be very difficult to recover $S$ from $(E,P)$. Furthermore, extending the design with row space projectors can make the problem even harder. The general system is defined as follows:
\begin{align}\label{eq:CSRSP-MF-keys}
\begin{aligned}
S = (\mathbf{F},\mathbf{H}^*)
\quad\mathrm{and}\quad
P = ((\mathbf{Y}^*_{\mathbf{B}_1},\mathbf{Y}^*_{\mathbf{B}_2}\mathbf{F}),
(\mathbf{X}_{\mathbf{D}_1},\mathbf{H}^*\mathbf{X}_{\mathbf{D}_2})),
\end{aligned}
\end{align}
\begin{align}\label{eq:CSRSP-MF-encdec}
\begin{aligned}
E(P,\mathbf{m}) = 
\mathbf{Y}^*_{\mathbf{B}_1}\mathbf{m}\mathbf{X}_{\mathbf{D}_1} 
+ \mathbf{w} 
- \mathbf{Y}^*_{\mathbf{B}_2}\mathbf{F}\mathbf{w}\mathbf{H}^*\mathbf{X}_{\mathbf{D}_2}
\quad\mathrm{and}\quad
D(S,\mathbf{c}) = \mathbf{F}\mathbf{c}\mathbf{H}^*.
\end{aligned}
\end{align}

There is one more intriguing observation to make. The designed system seems related to the McEliece public-key encryption scheme based on error-correcting codes \cite{mceliece1978public,menezes2018handbook}. The McEliece scheme is defined as follows:
\begin{align}\label{eq:McEliece:keys}
\begin{aligned}
S = (\mathbf{S},\mathbf{G},\mathbf{P})
\quad\mathrm{and}\quad
P = (\mathbf{S}\mathbf{G}\mathbf{P},t),
\end{aligned}
\end{align}
\begin{align}\label{eq:McEliece:encdec}
\begin{aligned}
E(P,\mathbf{m}) = \mathbf{m}\mathbf{S}\mathbf{G}\mathbf{P}+\mathbf{z}
\quad\mathrm{and}\quad
D(S,\mathbf{c}) = H(\mathbf{c}\mathbf{P}^T\mathbf{G})\mathbf{S}^{-1},
\end{aligned}
\end{align}
where $\mathbf{G}$ is an $k$ by $n$ generator matrix for a binary $(n,k)$-linear code that can correct $t$ errors (see e.g. MacWilliams and Sloane \cite{macwilliams1977theory}), $H$ is an efficient decoding algorithm for the code generated by $\mathbf{G}$, $\mathbf{z}$ is a~random binary error vector of length $n$ containing at most $t$ nonzero elements (ones) and injected into the ciphertext, $\mathbf{S}$ is a non-singular $k$ by $k$ matrix, and $\mathbf{P}$ is an $n$ by $n$ permutation matrix. 

The decoding function is designed to give output $H(\mathbf{c}\mathbf{P}^T|\mathbf{G}) = \mathbf{m}\mathbf{S}$. That happens when the binary message $\mathbf{m}$ is separated from the error vector $\mathbf{z}\mathbf{P}^T$ by the error-correcting mechanism. For a linear code, it is enough to generate messages from the nullspace of the parity-check matrix $\mathbf{H} = [-\mathbf{G}_0^T|\mathbf{I}_{n-k}]$, which must be related to $\mathbf{G} = [\mathbf{I}_k|\mathbf{G}_0]$. Since $\mathbf{m}\mathbf{S}\mathbf{G}\mathbf{H}^T = \mathbf{0}$, we have
$\mathbf{c}\mathbf{P}^T\mathbf{H}^T = \mathbf{z}\mathbf{P}^T\mathbf{H}^T$, which next can be used to calculate $\mathbf{m}\mathbf{S}$ with a~predefined error-correcting mechanism. As a~result, given the non-singular $\mathbf{S}$, we have $\mathbf{m}\mathbf{S}\mathbf{S}^{-1} = \mathbf{m}$ and the ciphertext is correctly decrypted. The key step is to correct errors in $\hat{\mathbf{c}} = (\mathbf{m}\mathbf{S})\mathbf{G}+\mathbf{z}\mathbf{P}^T$ to obtain $\mathbf{m}\mathbf{S}$. That is the role of the mapping $H$ introducing essential security measures. 

Also, notice that given the full rank matrix $\mathbf{G}$, we have $\mathbf{G}\mathbf{G}^+ = \mathbf{I}_k$ and for non-singular $\mathbf{S}$ we get
\begin{align}
(\mathbf{S}\mathbf{G}\mathbf{P})
(\mathbf{S}\mathbf{G}\mathbf{P})^+ = 
(\mathbf{S}\mathbf{G}\mathbf{P})
(\mathbf{P}^T\mathbf{G}^{+}\mathbf{S}^{-1}) = \mathbf{I}_k.
\end{align}
Therefore, $\mathbf{c}(\mathbf{S}\mathbf{G}\mathbf{P})^+ = \mathbf{m}+\mathbf{z}(\mathbf{P}^T\mathbf{G}^{+}\mathbf{S}^{-1})$.
As before, we can hide the probability distribution of messages under the probability distribution of the random error vector $\mathbf{z}$, thereby making the attacks based on frequency analysis ineffective. 

The security of the McEliece scheme depends critically on the complexity of decomposing $P = \mathbf{S}\mathbf{G}\mathbf{P}$ and calculating $\mathbf{G}^+$. To the author's best knowledge, an effective attack on that scheme (typically enhanced with the Goppa error-correcting model) is not currently known (see e.g. Engelbert et al. \cite{engelbert2007summary} and Xagawa et al. \cite{xagawa2021fault}).

As we can see, systems (\ref{eq:McEliece:keys})-(\ref{eq:McEliece:encdec}) and (\ref{eq:CSP-MF-keys})-(\ref{eq:CSP-MF-dec}) seem to share the same security properties. However, we can also argue the idea behind the McEliece scheme can be derived from the general design pattern presented in this paper. The McEliece scheme emerges when we exchange the role of $\mathbf{m}$ and $\mathbf{w}$ in (\ref{eq:CSP-MF-keys})-(\ref{eq:CSP-MF-dec}) and introduce error-correction for the code generating dictionary. From that point of view, system (\ref{eq:CSP-MF-keys})-(\ref{eq:CSP-MF-dec}), as well as its generalized version (\ref{eq:CSRSP-MF-keys})-(\ref{eq:CSRSP-MF-encdec}), define a~family of cryptographic functions.

The arguments above show how security depends on the secret key's secrecy, following Kerckhoffs' principle. In particular, in the case considered, the security depends on the complexity of finding generalized inverses for the encryption functions and immunity of the systems to the frequency analysis attacks. At this point, I leave the rigorous security level assessment open for further investigation.

\section{Reflections and closing remarks}

Decomposing a matrix into a product of interrelated components is a wonderful way to understand how a matrix works. The benefits coming from that understanding are well-known and proven by numerous applications. These include statistics, control, signal processing, optimization, and deep learning. But I also think there is more that we can learn from all that we have said above. And that is about what understanding means.

Let me explain with a wonderful and famous quotation by Stefan Banach:

\textit{A mathematician is a person who can find analogies between theorems; a better mathematician is one who can see analogies between proofs and the best mathematician can notice analogies between theories. One can imagine that the ultimate mathematician is one who can see analogies between analogies.}

I believe that it is what understanding is all about. To understand is to see that what we are looking at is something else in disguise. Isn't that what matrix factorizations show us?

Extracting key data components stored as numbers in a matrix, identifying column spaces and nullspaces, singular values, and the rank, all that explains how the matrix works. We can see each relationship between the numbers in detail. And the big picture, in which the bases for subspaces explain the way the subspaces interact. Here we see the column and row space projectors and (left) nullspace projectors. In this paper, we have taken a closer look at the structure of these projectors and shown that they are hidden in the factors of every matrix. Every matrix factorization is a projecting process in disguise.

I would like to share one more reflection. We are facing a revolution in artificial intelligence. The celebrated large language models are beginning to achieve goals and abilities we have never seen or thought possible. However, so far, we have not been able to understand why. Explaining artificial intelligence is one of the most fascinating open problems. What we know is that we can generate texts and images at an impressive human level or beyond with multilayered compositions of functions with billions of billions of parameters optimized in a long learning process with enormous training datasets. In contrast, we only need one parameter to explain and understand how the (almost) entire visible universe works. That is the cosmological constant of the Einstein field equations in the general relativity theory. We can compress the universe into the form of a single equation. What we see around us seems to be something deeply hidden in disguise. Isn't that beautiful?

% \section*{Acknowledgments}
% This paper is partially based on the guest lecture (18.065 course) the author gave during his stay at the MIT Mathematics Department in Spring 2022 and 2023.

\bibliographystyle{hplain}
\bibliography{references}
\end{document}